\documentclass[11pt, oneside,dvipsnames]{amsart}

\newcommand{\T}{\mathfrak{T}}

\newcommand{\e}{\epsilon}

\renewcommand{\to}{\rightarrow}

\newcommand{\Z}{\mathbb{Z}}

\newcommand{\C}{\mathbb{C}}
\newcommand{\R}{\mathbb{R}}
\newcommand{\del}{\partial}

\renewcommand{\P}{\mathbb{P}}

\newcommand{\sub}{\subset}

\newcommand{\inv}{^{-1}}

\newcommand{\wt}{\widetilde}

\renewcommand{\to}{\rightarrow}
\renewcommand{\P}{\mathbb{P}}

\newcommand{\ol}{\overline}

\newcommand{\im}{\normalfont\text{im}}
\DeclareMathOperator{\newt}{Newt}
\newcommand{\spn}[1]{\text{span} \langle {#1} \rangle }

\newcommand{\Conv}{\normalfont\text{Conv}}

\newcommand{\diff}{\normalfont\text{Diff}}

\newcommand{\hkra}{\hookrightarrow}

\newcommand{\set}[1]{\left\{ #1 \right\}}

\newcommand{\cc}[1]{\overline{#1}}

\newcommand{\sm}{\ensuremath{\setminus}}

\newcommand{\s}[2]{\sum\limits_{#1}{#2} }

\newcommand{\cl}{\colon}

\newcommand{\lbr}[1]{\Bigl(#1\Bigr)}

\newcommand{\lspan}[1]{\langle {#1}\rangle}

\newcommand{\lc}{\underline}

\newcommand{\bC}{\mathbb{C}}
\newcommand{\bD}{\mathbb{D}}

\newcommand{\bP}{\mathbb{P}}
\newcommand{\bQ}{\mathbb{Q}}
\newcommand{\bR}{\mathbb{R}}

\newcommand{\bT}{\mathbb{T}}

\newcommand{\bZ}{\mathbb{Z}}

\newcommand{\cP}{\mathcal{P}}

\newcommand{\fT}{\mathfrak{T}}

\newcommand{\abc}{{(a,b,c)}}
\newcommand{\abcprime}{{(a',b',c')}}

\usepackage{tikz-cd}
\usepackage{enumerate}
\usepackage{setspace}
\usepackage[english]{babel}
\usepackage{mathtools}
\usepackage[margin=1in]{geometry}
\usepackage{amssymb}
\usepackage[alphabetic]{amsrefs}
\usepackage{amsmath}
\usepackage{extarrows}
\usepackage{microtype}
\usepackage{graphicx}
\graphicspath{ {./pics/} }
\usepackage[colorlinks=true, allcolors=blue]{hyperref}
\usepackage{amsthm}
\usepackage{indentfirst}
\usepackage{import}
\usepackage{xifthen}
\usepackage{pdfpages}
\usepackage{transparent}
\usepackage{changepage}
\usepackage{wrapfig}
\hypersetup{
	colorlinks=true,
	linkcolor=Blue,
	filecolor=red,      
	urlcolor=red,
	citecolor=teal,
}

\newenvironment*{prooflemma*}{\paragraph{Proof:}}{}
\newtheorem{Theorem}{Theorem}[section]

\newtheorem{lemma}[Theorem]{Lemma}
\newtheorem{prop}[Theorem]{Proposition}
\newtheorem{cor}[Theorem]{Corollary}

\theoremstyle{definition}
\newtheorem{remark}[Theorem]{Remark}

\newtheorem{defn}[Theorem]{Definition}
\newtheorem{example}[Theorem]{Example}

\numberwithin{equation}{subsection}

\makeatother

\newcommand{\Addresses}{{
  \bigskip
  \footnotesize

  \textsc{Soham Chanda, Rutgers University}\\
 \indent\textit{E-mail address}: \texttt{sc1929@math.rutgers.edu}\par
  \indent\textsc{Amanda Hirschi, University of Cambridge}\\
  \indent\textit{E-mail address}: \texttt{aj616@cam.ac.uk}\par
    \textsc{Luya Wang, UC Berkeley}\\
   \indent \textit{E-mail address}: \texttt{luyawang@math.berkeley.edu}

}}

\title{Infinitely many monotone Lagrangian tori in higher projective spaces}
\author{Soham Chanda, Amanda Hirschi, Luya Wang}

\begin{document}

\maketitle
\begin{abstract}
   In \cite{Via16}, Vianna constructed infinitely many exotic Lagrangian tori in $\bP^2$. We lift these tori to higher-dimensional projective spaces and show that they remain non-symplectomorphic. Our proof is elementary except for an application of the wall-crossing formula of \cite{PT20}.
\end{abstract}

\tableofcontents
\setlength{\parskip}{0.5em}

\section{Introduction}
Given a symplectic manifold $(M,\omega)$, we are interested in studying its monotone Lagrangian tori up to symplectomorphisms of $M$. Here, \emph{monotone} means that the area homomorphism $\sigma: \pi_2(M, L) \to \R$ given by $[D]\mapsto \int_D \omega$ is a positive scalar of the Maslov homomorphism $\mu: \pi_2(M, L) \to \Z$ defined in \cite{Arnold}. We call a Lagrangian torus in $\P^n$ \emph{exotic} if it is monotone and not symplectomorphic to the standard Clifford torus. The earliest example of an exotic torus dates back to \cite{Chekanov1996}.

Recently, the use of almost toric fibrations has become an important tool in constructing new examples of Lagrangian tori. For example, Vianna has constructed infinitely many exotic tori in $\P^2$ \cite{Via16} and in del Pezzo surfaces \cite{Via17}. For more details on almost toric fibrations, see \cite{Symington, Leung_Symington, Evans_ATF}. For previous constructions of non-Hamiltonian isotopic Lagrangian tori in higher dimensions, see for example \cite{Auroux_Lagrangians, PT20} and \cite{Brendel}. 

 Our exotic examples are lifts $\ol T_\abc$ of the Vianna tori $T_\abc$ in $\P^2$. Here $\abc$ is a triple of natural numbers satisfying the Markov equation $a^2+b^2+c^2 = 3abc$. These tori are defined and reviewed in \textsection \ref{sec:lifting_vianna_tori}. Any exotic torus $T_\abc$ can be obtained from the Clifford torus in $\bP^2$ by a sequence of mutations. We show that the lifted Vianna tori can be obtained from the Clifford torus in $\bP^n$ by a sequence of \emph{solid mutations}, a generalisation of mutations to higher dimension. We study how the disk potentials change under a solid mutation, using as essential input a wall-crossing formula from \cite{PT20}. Similar to \cite{Via16}, we do not determine the disk potential explicitly. Instead, we show that the associated Newton polytope, defined in \textsection\ref{subsec:newt}, is uniquely determined by $\abc$.

\begin{Theorem}\label{thm:newt-pol-differ}
    The Newton polytope of the disk potential of $\ol T_\abc$ is a nondegenerate simplex in $\bR^n$. One $2$-dimensional face is a triangle with affine edge lengths $a$, $b$ and $c$, while the affine length of any other edge is $1$. In particular, if $\{a,b,c \} \neq \{a',b',c'\}$, then there is no symplectomorphism of $\P^n$ that maps the lifted Vianna torus  $\ol T_{\abc}$ to $\ol T_\abcprime$. 
\end{Theorem}

\begin{cor}\label{cor:exotic} For any $n \geq 3$, $\bP^n$ admits infinitely many distinct exotic Lagrangian tori. 
\end{cor}

\begin{remark}
    The upcoming work \cite{DTVW} also finds infinitely many non-symplectomorphic monotone Lagrangian tori in complex projective spaces, using a different method.
\end{remark}

\subsection*{Acknowledgements} S.C. thanks Chris Woodward, Yuhan Sun, Dennis Auroux and Georgios Dimitroglou Rizell for valuable discussions. A.H. is grateful to her advisor Ailsa Keating and Jeff Hicks for valuable discussions and explanations and Noah Porcelli for his patience. L.W. thanks Jo\'{e} Brendel for helpful conversations.

The authors met at the PIMS summer school on Floer homotopy theory and while this project does not directly further that topic, we want to thank the organisers and the institute for that opportunity. 
A.H. is supported by an EPSRC scholarship. L.W. acknowledges support by the NSF GRFP under Grant DGE 2146752.
\section{Geometric Preliminaries}

In this section we define the necessary geometric constructions in order to apply \cite[Theorem 1.1]{PT20} in \textsection\ref{sec:wall-cross}. We introduce the notion of a solid mutation configuration and solid mutations, generalising mutations of a $2$-dimensional Lagrangian torus along a disk to higher dimensions. Subsequently, we define the lifts of the Vianna tori and show that they are related by solid mutations.

\subsection{Solid mutation configurations}\label{subsec:smc} We generalise the results of \cite[\textsection4.4, \textsection4.5]{PT20} to higher dimensions. Compare with \cite[\textsection5.3]{PT20}, where the ambient manifolds are required to be toric. In particular, the definition of solid mutation matches with the definition of higher mutation in \cite{PT20} with mutation configuration $(F,w)$ where $F$ is an $n-1$ dimensional face of a moment polytope.

\begin{defn}\label{def:solid-mut-config} Let $(M,\omega)$ be a symplectic manifold. A pair $(L,\T)$ is a \emph{solid mutation configuration (SMC)} in $M$ if 
    \begin{itemize}
        \item $L$ is a Lagrangian torus;
        \item $\T$ is a Lagrangian solid torus, i.e. $\T$ is diffeomorphic to $\bD^2 \times \bT^{n-2}$;
        \item $L$ and $\T$ intersect cleanly along the boundary of $\T$;
        \item\label{rigid-torus} the inclusion $\fT\cap L \hookrightarrow L$ is (diffeomorphic to) an embedding of Lie groups.
    \end{itemize}
\end{defn}

Here two submanifolds $N_0$ and $N_1$ of $M$ \emph{intersect cleanly} if $K = N_0\cap N_1$ is a smooth submanifold of $M$ and $T_xK = T_xN_0 \cap T_xN_1$ for any $x \in K$.

Let us make the following observation.

\begin{lemma}\label{lemm:dondivisor} Suppose $(L,\fT)$ is a SMC in $(M,\omega)$ and $L$ is monotone. Then there exists a divisor $D\sub X\sm (L\cup \bT)$ Poincar\'e dual to $dc_1(X)$ for some $d \gg 1$, so that $L$ is exact.
\end{lemma}

\begin{proof} By \cite[Theorem 3.3]{PT20}, this higher dimensional analogue of \cite[Corollary 3.4]{PT20} is immediate.\end{proof}

As in \cite{PT20} we will construct a model neighbourhood for SMCs, which will allows us to define solid mutations. The key ingredient is a Weinstein neighbourhood theorem for SMCs (compare to \cite[Lemma 4.11]{PT20}), for which we need the following technical result.

\begin{lemma}\label{extend-diffeo} Suppose $\psi \cl [\frac12,1]\times\bT^n \to [\frac12,1]\times\bT^n$ is a diffeomorphism with $\psi(1,x) = (1,x)$ for $x \in \bT^n$. Then there exists $\epsilon >0$ and a diffeomorphism $\Psi \cl \bD\times \bT^{n-1}\to \bD\times\bT^{n-1}$ which agrees with $\psi$ on $[1-\epsilon,1]\times\bT^n$. If $\psi$ is equivariant with respect to a torus action on $\bT^n$, then we can choose $\Psi$ to be equivariant as well. 
\end{lemma}

\begin{proof} Write $\psi = (\psi',\psi'')$ and define $\psi_r(x) := \psi''(r,x)$ for $x \in \bT^n$. Then there is $0 <\epsilon <\frac19$ so that $\psi_r$ is a diffeomorphism for $|1-r|< 3\epsilon$. In particular, $\psi_r$ defines an (equivariant) isotopy from $\psi_{1-2\epsilon}$ to the identity. 
Let $\rho\cl [0,1]\to [0,1]$ be a smooth cutoff function with $\rho(t) = t$ for $t\geq 1-\epsilon$, $\rho\equiv 1-2\epsilon$ on $[0,1-2\epsilon]$ and $\rho'(t) > 0$ on $(1-2\epsilon,1]$. Similarly, let $\beta\cl[\frac12,1]\to [0,1]$ be a smooth cutoff function so that $\beta(t) = \frac12t$ near $\frac12$, $\beta < \frac 12$ on $[\frac12,1-2\epsilon]$, $\beta \equiv 1$ on $[1-\epsilon,1]$ and $\beta$ is strictly increasing on $[1-2\epsilon,1-\epsilon]$. Define $\psi^{(1)}\cl [\frac12,1]\times \bT^n\to (0,1]\times \bT^n$
by $$\psi^{(1)}(t,x) = (\beta(t)\psi'(\rho(t),x),\psi''(\rho(t),x)).$$ 
This is a diffeomorphism by the choice of $\rho$ and $\beta$ and agrees with $\psi$ near $\{1\}\times \bT^n$.\par
It suffices thus to extend the closed embedding 
$$\psi^{(1)}\cl [\frac12,r]\times \bT^n \to [0,1]\times\bT^n\sm \psi^{(1)}\lbr{(r,1])\times \bT^n}$$
for some $r < 1-2\epsilon$. We can write it as 
$$\psi^{(1)}(t,x) = (t h(x),\varphi(x))$$ for $t\in [\frac{1}{2},r]$, where $\varphi\in \diff(\bT^n)$ is (equivariantly) isotopic to the identity and $h \cl \bT^n \to (0,1]$ is smooth. Note that $\im(4h)\times \bT^n$ is the inner boundary of $\psi([1-2\epsilon,1]\times\bT^n)$. Given an (equivariant) isotopy $\{\varphi_s\}_{s\in [0,1]}$ from the identity to $\varphi$, let $\chi\cl [0,1/3]\to [0,1]$ be a smooth cutoff function, so that $\chi \equiv 0$ near $0$ and $\chi \equiv 1$ near $\frac13$. Fix also a smooth cutoff function $\eta \cl [0,r]\to [0,1]$ so that $\eta \equiv 0$ near $\frac13$, $\eta\equiv 1$ on $[\frac12,r]$ and $\eta'(t) \geq 0$. Set $a := \min h$. Then define $\tilde{\Psi}\cl [0,1]\times \bT^n\to [0,1]\times \bT^n$ by 
\begin{equation*} \tilde{\Psi}(t,x) = \begin{cases}
\psi^{(1)}(t,x) \quad & r \leq t \leq 1\\
\lbr{at\,(\frac{h(x)}{a})^{\eta(t)},\varphi(x)}\quad & \frac13\leq t \leq r\\
(ta,\varphi_{\chi(t)}(x))\quad & 0\leq t\leq \frac13
\end{cases}\end{equation*}
This descends to the desired diffeomorphism of the solid torus. If we start with an equivariant $\psi$, then $\tilde{\Psi}$ is equivariant by construction, and thus so is $\Psi$.\end{proof}

\begin{lemma}[Weinstein neighborhood theorem for solid mutation configurations]\label{lem:compare-solid-mutation-config} Suppose $(L_i,\fT_i) \sub (M_i,\omega_i)$ is a solid mutation configuration for $i \in \{0,1\}$ with $\dim(M_0)= \dim(M_1)$. Then there exist neighbourhoods $U_i \sub M_i$ of $L_i\cup \fT_i$ and a symplectomorphism $\psi \cl U_0 \to U_1$ mapping $(L_0,\fT_0)$ to $(L_1,\fT_1)$.
\end{lemma}

\begin{proof}  Fix a basis $v_{i,1},\dots,v_{i,n-1}$ of the Lie algebra of $\fT_i\cap L_i$ and extend it to a basis $v_{i,1},\dots,v_{i,n}$ of the Lie algebra of $L_i$. Let $\phi \cl L_0 \to L_1$ to be the unique morphism of Lie groups with $d\phi(e)v_{i,0} = v_{i,1}$. Then $d\phi(e)$ maps the Lie algebra of $L_0 \cap \fT_0$ to the one of $L_1 \cap \fT_1$, so $\phi(L_0 \cap \fT_0) = L_1\cap\fT_1$. Extend $\phi$ to a symplectic bundle isomorphism $\Phi\cl TM_0|_{L_0}\to TM_1|_{L_1}$, which maps $T\fT_0|_{\del \fT_0}$ to $T\fT_1|_{\del \fT_1}$.

By the Weinstein neighbourhood theorem, $\Phi$ defines a symplectomorphism $\psi'\cl U'_0 \to U'_1$ for neighbourhoods $U'_0$ and $U'_1$ of $L_0$, respectively $L_1$. Then $\psi'$ maps $U'_0 \cap \fT_0$ to a manifold tangent to $U'_1 \cap \fT_1$ with the same boundary. As cleanly intersecting submanifolds admits a nice normal form near their intersection, we may change $\psi'$ by a Hamiltonian isotopy to assume $\psi'$ maps $U'_0 \cap \fT_0$ to $U'_1\cap \fT_1$. By Lemma \ref{extend-diffeo} we can extend its restriction to $U_0'\cap \fT_0$ to a diffeomorphism $\varphi\cl \fT_0 \to \fT_1$, possibly shrinking $U_i'$. As $\varphi$ is induced by $\Phi$ near $\del \fT_0$, the standard lift of $\varphi$ to a symplectic vector bundle isomorphism $TM_0|_{\fT_0}\to TM_1|_{\fT_1}$ extends $\Phi$. Using the Weinstein neighbourhood theorem again, we obtain the desired symplectomorphism $\psi$.
\end{proof}

Suppose a torus $\bT$ embeds as subtorus of $\del \fT_0$ and $\del \fT_1$. Multiplication by elements of $\bT$ induces a torus action on $L_i$ and $\fT_i$ which extends to a Hamiltonian $\bT$-action on a neighbourhood of $L_i\cup \fT_i$. Using the equivariant version of the Weinstein neighbourhood theorem and Lemma \ref{extend-diffeo}, we can choose $\psi$ in the previous statement to be equivariant.\par  

We now construct our model neighbourhood, see also \cite[\textsection4.5]{PT20}. Set $X_1 := \bC^2\sm\{x_1x_2 = 1\}$ and endow it with the Lefschetz fibration $\pi \cl X_1 \to \bC \sm \{1\}$ defined by $\pi(x) = x_1x_2$. Given a simple loop (or more generally an embedded path) $\gamma$ in $\bC\sm\{0,1\}$ define the Lagrangian 
\begin{equation}
\label{eq:path_to_lagrangian}
    T_\gamma:= \set{ (x,y) \in \C^2  \bigg| |x| = |y|, \; \pi(x,y) \in \im(\gamma)}.
\end{equation}
By \cite[Lemma 4.13]{PT20}, there exists a primitive $\theta$ of $\omega_{\normalfont\text{std}}|_{X_1}$ and $A > 0$ so that $T_\gamma$ is exact with respect to $\theta$ if and only if $\gamma$ encloses $1$ and a disk of area $A$. Given such a loop $\gamma$ we say $T_\gamma$ (or just $\gamma$) is of \emph{Clifford type} if $\gamma$ encloses $0$ and of \emph{Chekanov type} otherwise. 
If $\ell$ is a line segment starting at $0$, then $T_\ell$ is a vanishing cycle, in particular, a Lagrangian disk.

Given $n \geq 3$, set $X := X_1\times \bC^{n-2}$ and endow it with the restriction of the standard symplectic form on $\bC^n$. Denote $\cc{T}_{\gamma} := T_\gamma\times \bT^{n-2}$. This is exact, with respect to $\tilde{\theta}:= \theta \oplus \theta_{n-2}$ for the standard primitive $\theta_{n-2}$ of $\omega_{\normalfont\text{std}}$ on $\bC^{n-2}$, if $T_\gamma$ is exact. The following is a straightfoward exercise.

\begin{lemma}\label{lem:model-smc} Suppose $\gamma$ is a loop enclosing both $0$ and $r$ and let $\ell$ be the line segment from $0$ to $\min(\{\im(\gamma)\cap i\bR\})$. Then $(\cc{T}_\gamma,\cc{T}_\ell)$ is an SMC.
\end{lemma}

We obtain the following corresponding generalisation of \cite[Lemma 4.17]{PT20} by applying Lemma \ref{lem:compare-solid-mutation-config} and the discussion afterwards. 

\begin{cor}\label{cor:local-model-smc}
If $(L,\fT)$ is an SMC in $(M^{2n},\omega)$, there exists a neighbourhood $U \sub M$ of $L\cup \fT$ and an equivariant symplectic embedding $\psi \cl U \hkra X_1\times \bC^{n-2}$ so that $\psi(L,\fT) = (\cc{T}_\gamma,\cc{T}_\ell)$ for $\gamma$ of Clifford type and $\ell$ a line segment as above. 
\end{cor}

By \cite[Lemma 4.14]{PT20}, $\cc{T}_\gamma$ and $\cc{T}_{\gamma'}$ are isotopic through a compactly supported Hamiltonian if and only if $\gamma$ and $\gamma'$ (as above) are smoothly isotopic in $\bC\sm\{0,1\}$ and enclose disks of the same area. Thus the following definition is well-defined up to Hamiltonian isotopy.

\begin{defn} Let $(L,\fT)$ be an SMC in $(M^{2n},\omega)$ and let $\psi$ be a symplectomorphism as in Corollary \ref{cor:local-model-smc}. The \emph{solid mutation of $L$ along $\fT$} is $L_\fT := \psi\inv(\cc{T}_{\gamma'})$ for any simple loop $\gamma'$ in $\psi(U)$ of Chekanov type.
\end{defn}

\begin{lemma}\label{lem:smc-monotone}
    If $(L,\fT)$ is an SMC in $(M,\omega)$ and $L$ is monotone, then so is $L_\fT$.
\end{lemma}

\begin{proof} Using Corollary \ref{cor:local-model-smc}, the proof is analogous to the proof of \cite[Lemma 2.5]{Cha23}.\end{proof}

\subsection{Lifting Vianna tori}
\label{sec:lifting_vianna_tori}
A \emph{Markov triple} $(a,b,c)$ is a triple of positive integers satisfying the Diophantine equation $a^2+b^2+c^2 = 3abc$. The set of these triples forms the vertices of the \emph{Markov tree}, where $(a,b,c)$ is connected to $(a',b',c')$ by an edge if and only if $(a',b',c')$ is a \emph{Markov mutation} of $(a,b,c)$, i.e., of the form $(3bc-a,b,c)$, $(a,3ac-b,c)$ or $(a,b,3ab-c)$. By an elementary argument involving Vieta jumping, one sees that the Markov tree is indeed connected and infinite.

For any Markov triple $(a,b,c)$, \cite{Via16} constructs a monotone Lagrangian torus $T_{(a,b,c)}$ in $\P^2$, whose Hamiltonian isotopy class is uniquely determined by $\abc$. To make this precise, let $\P(a^2, b^2, c^2)$ be the weighted projective space associated to a Markov triple $\abc$ with associated degenerations from $\P^2$. By \cite{Via16}, $\P^2$ can be obtained from $\P^2(a^2,b^2,c^2)$ by performing at most three rational blow-downs.

\begin{defn}[Vianna tori]
    The \emph{Vianna torus} $T_{\abc}$ \emph{associated to a Markov triple $\abc$} is defined to be the central fiber of the almost toric fibration of $\P^2$ obtained from the rational blow-down of $\P^2(a^2,b^2,c^2)$.\footnote{Note that Vianna denotes $T_\abc$ by $T(a^2,b^2,c^2)$ instead.}
\end{defn}

In \cite{Via16}, he shows that mutating $T_\abc$ results in a torus which is Hamiltonian isotopic to $T_\abcprime$, where $\abcprime$ is a Markov mutation of $\abc$.

We will lift these tori to monotone Lagrangian tori in $\bP^n$ for $n \geq 3$ using symplectic reduction. In Lemma \ref{lem:mutation-lifts}, we show that a mutation of $T_\abc$ corresponds to a solid mutation of its lift in higher dimensions.

Fix thus $n \geq 3$ and let 
$$\mu_n \cl \bP^n \to \bR^{n-2}: [z]\mapsto \frac{1}{|z|^2}(|z_3|^2,\dots,|z_n|^2)$$ 
be the moment map of the standard Hamiltonian $\bT^{n-2}$-action acting on the last $n-2$ homogeneous coordinates. In particular, it acts freely on $F_n:= \mu\inv_n\lbr{\set{\frac{1}{n+1}\sum_{i=1}^{n-2}e_i}}$. A computation in local coordinates shows that $F_n/\bT^{n-2}$ equipped with the reduced symplectic form is symplectomorphic to $\bP^2$.

Let $q \cl F_n \to\bP^2$ be induced by the quotient map. A straightforward computation shows that the preimage of the Clifford torus in $\bP^2$ under $q$ is the Clifford torus in $\bP^n$. 
\begin{defn}Given a Markov triple $\abc$, we define the \emph{lifted Vianna torus} to be $$T^{(n)}_\abc := q\inv(T_\abc).$$
\end{defn} 

We see that $T^{(n)}_\abc$ is a Lagrangian torus by the definition of the symplectic structure on the symplectic reduction. If $n$ is clear from the context, we let $\ol T_\abc := T^{(n)}_\abc$.

\begin{remark}\label{rem:H1basis} As we can also do the reduction inductively, at each step reducing by an $S^1$-action, we see that for each $\abc$, we obtain a tower $T^{(n)}_\abc =: T_n\to T_{n-1}\to \dots \to T_2 := T_\abc$, where $T_{j+1}\to T_j$ is the restriction of a principal $S^1$-bundle on $\bP^j$. As the Euler class of this $S^1$-bundle is a multiple of $[\omega_{\text{FS}}]$, its restriction to $T_j$ vanishes. Thus $T^{(n)}_\abc\to T_\abc$ is a trivial $\bT^{n-2}$-bundle.\end{remark}

\begin{lemma}\label{lem:mutation-lifts}
    If $\abcprime$ is a Markov mutation of $\abc$, then $T^{(n)}_{\abc}$ and $T^{(n)}_\abcprime$ are solid mutations of each other. In particular, each $T^{(n)}_\abc$ is monotone.
\end{lemma}

\begin{proof} By Vianna's construction and \cite[Lemma 4.21]{PT20}, there exists a mutation configuration $(T_{\abc},D)$ in $\bP^2$, so that the associated mutation of $T_{\abc}$ is $T_{\abcprime}$. Let $\T = q\inv(D)$. As $q$ is the quotient map of a free $\bT^{n-2}$-action, $(T^{(n)}_{\abc},\T)$ is a solid mutation configuration in $\bP^n$. It follows from the local model in \cite{PT20}, respectively \textsection\ref{subsec:smc} that the solid-mutation of $T^{(n)}_{\abc}$ along $\T$ is the lift of the mutation of $T_{\abc}$ along $D$, and thus Hamiltonian isotopic to $T^{(n)}_{\abcprime}$. The last assertion follows from Lemma \ref{lem:smc-monotone}.
\end{proof}

\section{A Wall-Crossing Formula for the Lifted Vianna Tori}\label{sec:wall-cross}

In this section, we explain how to obtain the wall-crossing formula for the disk potential under a solid mutation. This is a variation of \cite[Theorem 5.7]{PT20}, where we do not require the toric assumption due to our Weinstein neighbourhood theorem for general solid mutation configurations.

Let $X_1 := \C^2 \setminus \{ x_1x_2=1 \}$ be as in the previous section and let $\gamma$ be a loop of Clifford type in $X_1$. Suppose $\ell$ is a straight line segment in $\bC^*$ joining the origin to a point $p\in \gamma$ and only intersecting $\gamma$ at $p$. Set $L_0 := T_\gamma$ and $D_0 := T_\ell$ as defined in (\ref{eq:path_to_lagrangian}). By \cite[Lemma 4.15]{PT20}, there is a small neighbourhood $U_0$ of $L_0 \cup D_0$ such that $U_0$ is Liouville and the Liouville completion of $\ol{U_0}$ (which we can take to be a Liouville domain) agrees with Liouville completion of $X_1$, which is just $X_1$ itself. Let $L_1$ denote a Chekanov type torus in $U_0$.
 
Set $\ol L_i : =   L_i \times \bT^{n-2}$. This is an exact Lagrangian in $(U_0 \times (A_\e)^{n-2},\tilde{\theta}))$ for $i \in \{0,1\}$, where $A_\e:= \{ z \in \C |\, |z-1| <  \e\}$ for $\e > 0$ and $\wt\theta$ was defined in \textsection\ref{subsec:smc}. Let $U_s$ be a Liouville domain obtained by smoothing the corners in $\cc{U_0} \times (\cc{A_\e})^{n-2}$. Since the completion of $A_\e$ is $\C^*$, the completion of $U_s$ is $X_1 \times (\C^*)^{n-2}$, see \cite[\textsection 3.d.]{Oan06}.

Now we extract a local wall-crossing formula from \cite{sei13}. Given arbitrary local systems $\rho_i = (x_i,y_i,z_{ij}) \in (\C^*)^n$, where $i = 0,1$ and $j \in \{1,2,\dots, n-2 \}$, let $\bold L_i := (\ol L_i,\rho_i)$. 

\begin{lemma}[Local Wall-Crossing]\label{lemm:localwall}
    There exists a choice of basis of $H_1(\ol L_i;\bZ) \cong \Z^n $ such that
    $$HF_{X_1 \times (\C^*)^{n-2}} (\bold L_1, \bold L_2) \neq 0$$ 
    if and only if $(x_1,y_1,z_{1,1},z_{1,2}\dots , z_{1,(n-2)}) = (x_0,y_0(1+x_0),z_{0,1},\dots, z_{0,(n-2)})$. 
\end{lemma}
\begin{proof}
Let $(v_M,v^{(1)}_\bC, \dots, v^{(n-2)}_\bC)$ be a holomorphic strip in $X_1 \times (\bC^*)^{n-2} $ with boundary on $\bold L_0$ and $\bold L_1$. The component $v^{(i)}_\C$ is a holomorphic map from a strip to $\C^*$ with boundary on $\bT^{n-2}$, hence constant. Thus any strip between $\cc{L}_0$ and $\cc{L}_1$ in $M \times (\C^*)^{n-2}$ is a lift from a strip in $M$ between $L_0$ and $L_1$. The result follows from the same argument as in  \cite[Proposition 11.8]{sei13}.
\end{proof}

\begin{remark}\label{rem:basis-from-seeds}  \cite{Via16} exhibits $T_\abc$ as the central fibre of an almost toric fibration on $\bP^2$ over a base $B\sub \bR^2$. Thus, by \cite[Lemma 4.21]{PT20}, we have a canonical identification of $H_1(T_\abc;\bZ)$ with $\bZ^2$ defined by the integral affine structure on $B$. Any Lagrangian disk $D$ lying over the line segment in $B$ from a nodal point $p$ to the central point induces a unique basis of $H_1(T_\abc;\bZ)$ given by $\alpha := [\del D]$ and $\beta$, where $\beta\in \bZ^2$ is the primitive vector going in the direction of the ray. By \cite[Lemma 4.21]{PT20}, $\alpha$ and $\beta$ are orthogonal with respect to the inner product on $\bZ^2\sub \bR^2$. Mutating $T_\abc$ along $D$ gives us a torus in $\bP^2$ which is Hamiltonian isotopic to $T_\abcprime$ for some $\abcprime$ related to $\abc$ by a Markov mutation. Using a smooth isotopy from $T_\abc$ to $T_\abcprime$, $\alpha,\beta$ induce a basis of $H_1(T_\abcprime;\bZ)$ with respect to which 
\begin{equation}\label{eq:2dmut}
    W_{T_\abc} (x,y) = W_{T_\abcprime} (x,y(1+x)).
\end{equation}
holds. If $\alpha'$ and $\beta'$ are constructed in the same way using a different nodal point of the fibration (and different Lagrangian disk $D'$), then $\beta$ and $\beta'$ are linearly independent.\end{remark}

\begin{remark}\label{rem:sol-mut-basis} Given a basis $\alpha,\beta$ of $H_1(T_\abc;\bZ)$ we can use Remark \ref{rem:H1basis} to lift this to a basis $\wt\alpha,\wt\beta,\wt\gamma_1,\dots,\wt\gamma_{n-2}$ of $H_1(\ol T_\abc;\bZ)$, where $\wt\gamma_i$ is induced by the $S^1$-action of the $i^{\text{th}}$-factor of $\bT^{n-2}$. 
 \end{remark}  

\begin{Theorem}\label{thm:wallcrossing}
    There is a choice of basis of $H_1(\ol T_\abc;\bZ)$ and $H_1(\ol T_\abcprime;\bZ)$ such that the disk potentials are related by  \begin{equation}\label{eq:wall-cross-lifts}
        W_{\ol{T}_\abc} (x,y,z_1,\dots,z_{n-2}) = W_{\ol{T}_\abcprime} (x,y(1+x),z_1,\dots, z_{n-2}),
    \end{equation} 
    where $z_i$ is the holonomy around the $i^{\text{th}}$ factor of the $\bT^{n-2}$-fibre of the fibration $\ol T_\abc \to T_\abc$ and $x,y$ are the holonomy around horizontal lifts of the basis as in Remark \ref{rem:H1basis} for which the wall-crossing formula \eqref{eq:2dmut} holds.
\end{Theorem}

\begin{proof}  
    By Corollary \ref{cor:local-model-smc} there exists an equivariant symplectic embedding $\phi : U_0 \times (A_\e)^{n-2} \to \P^n$ and Lagrangians $L_0,D_0,L_1 \sub U_0$ as defined above so that $\phi(\ol{L}_0) = \ol{T}_\abc$, $\phi(D_0 \times \bT^{n-2}) = \T$ where $\T$ is a solid torus, together with a Hamiltonian diffeomorphism $\psi_H$ of $\bP^n$ satisfying $\phi(\cc{L}_1) = \psi_H (\ol{T}_\abcprime)$.\\
    Let $D\sub \bP^n$ be a Donaldson divisor for $(\ol T_\abc,\T)$ as in Lemma \ref{lemm:dondivisor}. Choosing $U_0$ and $\e$ sufficiently small, we may assume $\phi(U_0 \times (A_\e)^{n-2}) \sub \P^n \setminus {D}$. 
    We claim that $\phi(U_0 \times (A_\e)^{n-2})$ is a Liouville subdomain of $\bP^n\sm D$, i.e., that $\lambda:= \tilde\theta - \phi^*\theta_{\P^n \setminus D}$ is exact, where $\theta_{\P^n\setminus D}$ is a primitive of $\omega_{\text{FS}}|_{\bP^n\sm D}$. Let $\alpha$ be any loop in $U_0\times (A_\e)^{n-2}$. As $U_0 \times (A_\e)^{n-2}$ deformation retracts onto $\cc{L}_0 \cup D_0 \times \bT^{n-2}$, we may assume $\alpha = \iota_*\alpha'$ for some $\alpha'\in \pi_1(\cc{L}_0)$ and $\iota\cl \cc{L}_0 \hkra U_0\times (A_\epsilon)^{n-2}$ the inclusion. Thus $$\int_\alpha \lambda = \int_{\alpha'}\iota^*\lambda = 0,$$
    since $\cc{L}_0$ is exact with respect to both primitive $1$-forms. Thus $\lambda$ is exact.
    
By Lemma \ref{lemm:localwall} and \cite[Theorem 3.1]{PT20}, we know that under the given choice of basis of $H_1 (\cc{L}_i;\bZ)$, 
$$HF_{U_s} (\bold L_0, \bold L_1)\cong HF_{X_1\times (\bC^*)^{n-2}} (\bold L_0, \bold L_1)  \neq 0$$ 
if and only if $\bold{L}_0$ is $\cc{L}_0$ equipped with the local system $(x,y,z_1,\dots,z_{n-2})$ and $\bold{L}_1$ is $\cc{L}_1$ equipped with $(x,y(1+x),z_1,\dots,z_{n-2})$. As the disk potential is invariant under Hamiltonian isotopy, we see that $W_{\ol T_\abcprime} = W_{\phi(\cc{L}_1)}$. Therefore \cite[Theorem 1.1]{PT20} implies $$W_{\ol T_\abc} (x,y,z_1,\dots, z_{n-2}) = W_{\ol T_\abcprime} (x,y(1+x),z_1, \dots, z_{n-2}).$$
\end{proof}   

\begin{remark} The proof of Theorem \ref{thm:wallcrossing} was phrased in terms of general solid mutation configurations and the statement actually holds in that generality.
\end{remark}

\section{Distinguishing the Tori in $\bP^n$ }

In order to distinguish the tori $T_\abc$, Vianna shows that the Newton polytope associated to the disk potential of $T_\abc$ is a triangle with edges of affine length $a$, $b$  and $c$. As the Newton polytope of $W_{T_\abc}$ is invariant of the symplectomorphism class of $T_\abc$, this implies that the symplectomorphism class of $T_\abc$ in $\P^2$ is uniquely determined by $\abc$.

We will use this result from \cite{Via16} to prove a similar result for the lifted Vianna tori. The argument uses induction on the Markov tree and properties of the Vianna tori and their Newton polytopes. 

\subsection{Newton polytopes}\label{subsec:newt}

Given $n \geq 3$ let $R:= \bR[x_1^\pm,\dots,x_n^\pm]$ be the ring of Laurent polynomials in $n$ variables. We identify the set of monomials in $R$ with $\bZ^n$ in the obvious way. The \emph{Newton polytope} of a Laurent polynomial $f = \s{k\in \bZ^n}{a_kx_1^{k_1}\cdots x_n^{k_n}}\in R$ is the the closed convex hull
$$\newt(f) := \normalfont\text{Conv}(\{k \in \bZ^n : a_k \neq 0\}).$$
This association is equivariant with respect to the $\normalfont\text{GL}(n,\bZ)$-action on $R$, defined in \cite[Remark 4.2]{PT20}, and the standard action on $\bR^n$. 

\begin{example}\label{ex:newt-cliff} Recall that the disk potential  of the Clifford torus in dimension $n$ is 
\begin{equation}\label{eq:cliff-newt}
    W_{\ol T_{(1,1,1)}}(x) = x_1+\dots +x_n + \frac{1}{x_1\cdots x_n}.
\end{equation}
 Thus $\newt(W_{\ol T_{(1,1,1)}}) = \Conv(e_1,\dots,e_n,-\sum_{i=1}^ne_i)$ is an $n$-simplex. Here $e_1,\dots,e_n$ denotes the standard basis of $\bR^n$.\footnote{Note that this Laurent polynomial is not yet mutable with respect to $(x_1,\dots,x_n) \mapsto (x_1,x_2(1+x_1),x_3,\dots,x_n)$; we first have to apply the coordinate change $(x_1,\dots,x_n) \mapsto (\frac{x_1}{x_2},x_2,\dots,x_n)$.}
\begin{center}\begin{figure}[ht]
   \centering
 \hbox{\hspace{5em}\vspace{0.1em}   
   
        \def\svgscale{1}
    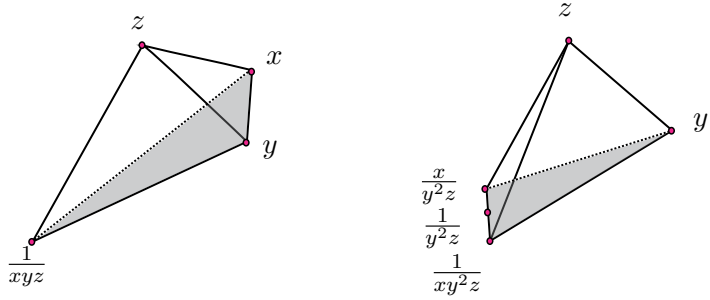

 }
 \caption{Newton polytopes of $W_{\ol T_{(1,1,1)}}$ and $W_{\ol T_{(1,1,2)}}$ }
 \label{fig:w3poly}
\end{figure}\end{center}
\end{example}

In particular, the Newton polytope of the Clifford torus is a \emph{Fano polytope} as defined in \cite{ACG12}; i.e.,
\begin{enumerate}
    \item\label{pol:nondeg} the polytope is convex;
    \item\label{pol:origin} it contains $0$ in its interior;
    \item\label{pol:vertices} its vertices are primitive in $\bZ^n$.
\end{enumerate}

In \textsection\ref{sec:wall-cross} we showed that a solid mutation of a suitable Lagrangian torus corresponds to a specific algebraic mutation of its disk potential. By \cite{ACG12}, an algebraic mutation of $f\in  R$ corresponds to a \emph{combinatorial mutation} of $\newt(f)$. Refer to \cite[Definition 5]{ACG12} for a precise description and to Remark 5 op. cit. for a discussion of the relationship between algebraic and combinatorial mutations. Note that we are only interested in algebraic mutations of the form $(x_1,\dots,x_n) \mapsto (x_1,x_2(1+x_1),x_3,\dots,x_n)$. See \cite[\textsection 5]{PT20} for more general algebraic mutations that occur if one uses different geometric mutations of Lagrangians.

\begin{example}\label{ex:newt-chek} Solid-mutating the Clifford torus $\cc{T}_{(1,1,1)}$ in $\bP^n$ once, we obtain the Chekanov torus $\cc{T}_{(1,1,2)}$ with disk potential 
$$W_{\ol T_{(1,1,2)}}(x_1,\dots,x_n) = x_2 + x_3+\dots +x_n+ \frac{(1+x_1)^2}{x_1x_2^2x_3\cdots x_n}.$$ 
See Figure \ref{fig:w3poly} for its Newton polytopes in dimension $2$ and $3$.
\end{example}

By \cite[Proposition 2]{ACG12}, the combinatorial mutation of a Fano polytope is again a Fano polytope. From Example \ref{ex:newt-cliff} it follows that $\newt(W_{\ol T_\abc})$ is a Fano polytope for any Markov triple $\abc$.
The proof of Theorem \ref{thm:newt-pol-differ} relies heavily on going back and forth between Laurent polynomials and their associated Newton polytopes. 

From now on we will use $x,y,z_1,\dots,z_{n-2}$ instead of $x_1,\dots,x_n$ to distinguish between the variables which are involved in the algebraic mutation and those which are not. We denote the associated coordinates of $\bR^n$ by $\textbf{x},\textbf{y},\textbf{z}_1,\dots,\textbf{z}_{n-2}$.

Now we prove some lemmas in preparation for the proof of the main theorem in \textsection\ref{sec:induction}.

\begin{lemma}\label{lem:observations}
\begin{enumerate}[\normalfont a)]
    \item\label{newt-triangles} Suppose the Newton polytope of 
    $$f(x,y) = \sum_{i= m}^M y^if_i(x)$$ 
    is Fano with $f_m \neq 0$ and $f_M \neq 0$. Then $m < 0 < M$. Moreover, if $\newt(f)$ is a triangle, then either $f_M$ or $f_m$ is a monomial.
    \item\label{newt-linearity} Suppose the Newton polytope of $$g(x,y,z) = \sum_{i = m}^M y^iz^{j_{1i}}_1\cdots z^{j_{(n-2)i}}_{n-2}g_i(x)$$ 
    is contained in the affine plane $H = w+\spn{v,e_1}$ for some $v = (v_1,\dots,v_n) \in \bZ^n$ such that $v_2 \neq 0$ and $w \in \bZ^n$. Then $j_{ri}$ depends linearly on $i$ for any $r \in \{1,\dots,n-2\}$.
    \end{enumerate}
\end{lemma}

\begin{proof}
\begin{enumerate}[\normalfont a)]
    \item By \eqref{pol:origin}, $\newt(f)$ has nonempty interior, so the first property is immediate. To see the second, suppose $f_M$ is not a monomial. Then $\newt(f) \cap \{\textbf{y} = M\}$ is an edge of $\newt(f)$; in particular it contains two vertices of $\newt(f)$. As $\newt(f)\cap \{\textbf{y} = m\}$ is nonempty and only touches the boundary of $\newt(f)$ it must therefore be a single point and thus $f_m$ a monomial.
    \item Any monomial in $g$ is of the form $x^ky^iz^{j_{1i}}_1\cdots z^{j_{(n-2)i}}_{n-2}$ where 
    $$(k,i,j_{1i},\dots,j_{(n-2)i}) = w + av + be_1 = (w_1+av_1+b,av_2+w_2,av'_{1}+w_1',\dots,av'_{n-2}+w'_{n-2})$$ 
    for some $a,b \in \bR$. As $v_2 \neq 0$, we obtain $j_{ri} =(i-w_2)\frac{v'_r}{v_2}+ w'_r$. 
\end{enumerate}    
\end{proof}

We will use the following properties of the Newton polytopes associated to the Vianna tori. The first result is a summary of \cite[Lemma 4.11]{Via16} and the discussion loc. cit.

\begin{lemma}\label{lemm:viatrian}
    The Newton polytope $\newt(W_{T_\abc})$ is a triangle whose edges have affine lengths $a,b,$ and $c$. Moreover, the coefficients of the monomials in $W_{T_\abc}$ corresponding to the vertices of $\newt(W_{T_\abc})$ are $\pm 1$.
\end{lemma}

The following statement is probably known to experts, but we did not find a proof in the literature. We provide a proof for completeness.

\begin{lemma} \label{lem:binomedge}
  Each point on an edge of the triangle $\newt(W_{T_\abc})$ corresponds to a monomial in $W_{T_\abc}$. Moreover, its coefficient is binomial.
\end{lemma}

\begin{proof}
    Let $\abcprime$ be a Markov triple obtained from $\abc$ by a Markov mutation. Fix bases of $H_1(T_\abc;\bZ)$ and $H_1(T_\abcprime;\bZ)$ so that Theorem \ref{thm:wallcrossing} holds. Let 
    $$W_{T_\abc}(x,y) = \sum_{i=m}^{M} y^iC_i(x)$$ be the disk potential with respect to the chosen basis. As $W_{T_\abcprime}(x,y) = W_{T_\abc} (x,\frac{y}{1+x})$
    is a Laurent polynomial, we must have $(1+x)^M | C_M(x)$. Then $C_M(x)$ cannot be a monomial, so $C_m(x) = ax^{k_m}$ for some $k_m\in \Z$ by Lemma \ref{lem:observations}\eqref{newt-triangles}. By the wall-crossing formula and \eqref{newt-triangles} applied to $W_{T_\abcprime}$, we obtain that $\frac{C_M(x)}{(1+x)^M}$ is a monomial with the same coefficient as $y^MC_M(x)$ in $W_{T_\abc}$. By Lemma \ref{lemm:viatrian}, the coefficient of $\frac{y^M C_M(x)}{(1+x)^M}$ in $W_{T_\abcprime}$ is $\pm 1$, so 
    $$C_M(x) = \pm x^{k_M}(1+x)^M.$$ 
    Thus, the coefficients along the edge $e = \Conv(\{(i,M) : x^i\in C_M(x)\})$ corresponding to $y^MC_M(x)$ are binomial. In particular, the monomials corresponding to the vertices of $e$ have exactly the same coefficients, i.e., both are either $1$ or $-1$.  By (the last observation in) Remark \ref{rem:basis-from-seeds}, choosing a different basis of $H_1(T_\abc;\bZ)$ will exhibit a different side of $\newt(W_{T_\abc})$ as the face given by intersecting the plane $\{\textbf{y} = M'\}$, where $M'$ is the maximal degree of $y$ appearing in $W_{T_\abc}$ with respect to the new basis. Thus we may conclude. 
 \end{proof}
   
\subsection{Distinguishing tori by induction}
\label{sec:induction}

To prove Theorem \ref{thm:newt-pol-differ} , we use an induction based on the (infinite) Markov tree. Explicitly, we verify that $(1,1,1)$ has the desired property $\cP$ as the base step. Then, assuming $\cP$ is satisfied by any Markov triple $\abc$ of (graph) distance $d$ away from $(1,1,1)$, we prove that $\cP$ holds for an elementary mutation of $\abc.$

The following result is a more precise formulation of Theorem \ref{thm:newt-pol-differ}. Set $\lc{z} := \sum_{i=1}^{n-2}z_i$.

\begin{prop}\label{prop:induction} Let $\abc$ be a Markov triple. Given a basis of $H_1(\cc{T}_\abc;\bZ)$ as in Remark \ref{rem:sol-mut-basis}, the following holds:
\begin{enumerate}[\normalfont i)]
    \item\label{ind-triangle}   $\newt(W_{\ol T_ \abc} - \lc{z})  $ is a triangle;
    \item\label{ind-eval} $W_{\ol T_\abc}(x,y,1,\dots,1) = W_\abc (x,y) + n-2$;
    \item\label{ind-tetrahedron} $\newt(W_{\ol T_\abc})$ is a simplex with one $2$-dimensional face given by $\newt(\ol{W}_{T_\abc}-\lc{z})$ and the other vertices being $e_3,\dots,e_n$. Moreover, they are at affine unit length from all other vertices;
    \item \label{ind-length} the affine lengths of the edges of the triangle $\newt(W_{\ol T_\abc} - \lc{z})  $ are $a, b$ and $c$.
\end{enumerate}

\end{prop} 

\begin{proof} Abbreviate $W_\abc := W_{T_\abc}$ and $\newt(\abc) := \newt(W_{T_\abc})$ and similarly for $\ol T_\abc$.\\ By Example \ref{ex:newt-cliff} and Example \ref{ex:newt-chek}, this holds for $(1,1,1)$ and $(1,1,2)$. Let $\abcprime$ be a Markov triple at a distance $d+1$ from $(1,1,1)$ for $d \geq 1$. Assume it is connected to a triple $\abc$ at distance $d$ from $(1,1,1)$. Fix an admissible basis for $H_1(\cc{T}_\abc;\bZ)$ so that \eqref{ind-triangle}-\eqref{ind-length} hold for $\abc$. Write $W_{\ol \abc} = \sum_{i=m}^M y^i \wt C_i(x,z_1,\dots,z_{n-2})$, where $m < 0< M$ are the extremal degrees with which $y$ appears in $W_{\ol\abc}$. By \eqref{ind-triangle}
 and \eqref{ind-eval}, for each monomial $x^iy^j$ in $W_\abc$ there exists a unique $x^iy^jz^{k_{ij1}}_1\cdots z^{k_{ij(n-2)}}_{n-2}$ in $W_{\ol \abc}$ whose coefficient is given by the coefficient of $x^i y^i$ by \eqref{ind-eval}. The proof of Lemma \ref{lem:binomedge} shows that 
 $$\wt C_M(x,z_1,\dots,z_{n-2}) = \pm x^k\sum_{j=0}^{M} \binom{M}{j} x^jz^{k_{j1}}_1\cdots z^{k_{j(n-2)}}_{n-2}$$ 
 for some $k_{jr} := k_{Mjr}$. By \eqref{ind-eval} and the proof of Lemma \ref{lem:binomedge}, $\newt(W_{\ol\abc}-\lc{z}) \cap \{\textbf{y} = M \}$ is a line. 
 
In particular, $k_{jr}$ depends linearly on $j$, i.e. $k_{jr} = \ell_r j + c_r$ for some $\ell_r,c_r\in \bQ$. Hence
\begin{equation}\label{eq:blah}
    \wt C_M(x,z_1,\dots,z_{n-2}) = \pm x^kz_1^{c_1}\cdots z_{n-2}^{c_{n-2}}(1+xz_1^{\ell_1}\cdots z_{n-2}^{\ell_{n-2}})^M.
\end{equation}
    By our choice of basis, $W_{\ol \abc}$ is mutable, so $(1+x)^M$ divides $\wt C_M$. This implies that $\ell=0$ and $c_r \in \bZ$ . Thus $\newt(W_{\ol\abc}-\lc{z})$ is contained in the affine plane $(k,M,c_1, \dots, c_{n-2}) + \lspan{v,e_1}$ for some $v \in \bZ^n\sm\{0,e_1\}$. Lemma \ref{lem:observations}\eqref{newt-linearity} implies that the $\normalfont\textbf{z}_r$-coordinate of points in $\newt(W_{\ol \abc} -\lc{z})$ depends linearly on the $\normalfont\textbf{y}$-coordinate. Hence
    \begin{equation}\label{eq:fulldpformula}
        W_{\ol \abc} (x,y,z_1,\dots,z_{n-2}) = \lc{z} + \sum _ {i=  m }^M y^i z_1^{f_1(i)}\cdots z_{n-2}^{f_{n-2}(i)} C_i(x) 
    \end{equation} where $W_\abc(x,y) = \sum _ {i=  m }^M y^i  C_i(x),$ and each $f_r$ is a linear function. Equations \eqref{eq:fulldpformula} and \eqref{eq:wall-cross-lifts} imply that \eqref{ind-triangle} and \eqref{ind-eval} hold for $\abcprime$.
    As 
    $$\newt(W_{\ol\abcprime}) = \Conv(\{e_3,\dots,e_n,\Conv(\newt(W_{\ol \abcprime}-\lc{z})\})$$ 
    is Fano, the points $e_3,\dots,e_n$ are not contained in $\Conv(\newt(W_{\ol \abcprime}-\lc{z})$. This implies the first claim of \eqref{ind-tetrahedron}. To see the second claim, note that the vertices of $\newt(W_\abcprime)$ are primitive in $\bZ^2$ as $\newt(W_{\abcprime})$ is Fano. By \eqref{ind-eval}, this shows that any edge $e$ between $e_r$ and a vertex of $\newt(W_{\ol\abcprime}-\lc{z})$ is of the form $e = \{e_r+ t(v,v'): t \in [0,1]\}$, where $v$ is primitive in $\bZ^2$ and $v'\in \bZ^{n-2}$. Thus the affine length of $e$ is $1$, as is the affine length of the edge between $e_r$ and $e_{r'}$.
    
    Finally, by Lemma \ref{lem:binomedge} and \eqref{ind-eval}, any lattice point on an edge of $\newt(W_\abcprime)$ lifts to a unique lattice point on an edge of $\newt(W_{\ol\abcprime}-\lc{z})$. Therefore, the affine lengths of the edges of the two triangles agree. By Lemma \ref{lemm:viatrian}, this proves (\ref{ind-length}) and concludes the induction.
\end{proof}    

\subsection{Proof of Theorem \ref{thm:newt-pol-differ}} The first two assertions of Theorem \ref{thm:newt-pol-differ} are immediate from Proposition \ref{prop:induction}. We now show the last assertion. The \emph{boundary Maslov-2 convex hull} $\mho_{\ol T_\abc}$ of $\ol T_\abc$ is the convex hull of $\set{\del [u] \,|\, u \cl (\bD,S^1)\to (\bP^n,\ol T_\abc) \text{ holomorphic with } \mu(u) = 2}\sub \pi_1(\ol T_\abc)$. By \cite[Remark 4.5]{Via16} and the simply-connectedness of $\bP^n$, we can identify the Newton polytope of $W_{\ol T_\abc}$ with $\mho_{\ol T_\abc}$. As the latter is an invariant of the Lagrangian up to symplectomorphism by \cite[Corollary 4.3]{Via16}, we obtain that $\ol T_\abc$ is not symplectomorphic to $\ol T_\abcprime$ for $\abc \neq \abcprime$.

\bibliographystyle{amsplain}
\bibliography{bib}

\Addresses

\end{document}